\documentclass[12pt,leqno]{amsart}
\usepackage{graphicx}
\usepackage[utf8]{inputenc}
\usepackage[T1]{fontenc}
\usepackage{comment}

\usepackage{a4wide}
\usepackage{indentfirst}

\usepackage{amsmath,amssymb}
\usepackage{amsthm}
\usepackage{hyperref}

\usepackage[margin=1.2in]{geometry}
\usepackage{graphicx}
\usepackage{listings}
\usepackage[utf8]{inputenc}
\usepackage{appendix}
\numberwithin{equation}{section}
\newtheorem{teo}{Theorem}[section]
\newtheorem*{teo*}{Theorem}
\newtheorem*{prop*}{Proposition}
\newtheorem*{corol*}{Corollary}
\newtheorem{prop}[teo]{Proposition}

\newtheorem{corol}[teo]{Corollary}
\newtheorem{lema}[teo]{Lemma}
\newtheorem{defi}[teo]{Definition}

\theoremstyle{definition}

\newtheorem{rmk}{Remark}

\newcommand{\D}{\mathbb{D}}
\newcommand{\E}{\mathbb{E}}

\newcommand{\R}{\mathbb{R}}

\newcommand{\1}{\mathbf{1}}
\newcommand{\phy}{\varphi}
\newcommand{\eps}{\varepsilon}

\title[\textbf{ON THE POSITIVITY OF THE DENSITY OF SDDEs driven by a fbm}]{\textbf{ON THE POSITIVITY OF THE DENSITY OF STOCHASTIC DELAY DIFFERENTIAL EQUATIONS DRIVEN BY A FRACTIONAL BROWNIAN MOTION}}
\author{Òscar Burés, Carles Rovira}

\address{Òscar Burés, Facultat de Matemàtiques i Informàtica, Universitat de Barcelona.
Gran Via de les Corts Catalanes, 585, 08007 Barcelona.}
\email{oscar.bures.mogollon@gmail.com}

\address{Carles Rovira, Facultat de Matemàtiques i Informàtica, Universitat de Barcelona.
Gran Via de les Corts Catalanes, 585, 08007 Barcelona.}
\email{carles.rovira@ub.edu}
\date{\today}
\thanks{C. Rovira is supported by the grant PID2021-123733NB-I00 from SEIDI, Ministerio de Economia y Competividad.}
\begin{document}
\begin{abstract}
    In this paper, we consider a Stochastic Delay Differential Equation with constant delay $r > 0$ and, under the same conditions on the coefficients needed to ensure the smoothness of the density plus an ellipticity condition on the diffusion term, we prove that the density function of the solution is strictly positive in its support. In order to prove it, we give a Gaussian-type lower bound for the density of the solution combining the Nourdin and Viens' density bounding method together with Kohatsu-Higa's method.
\end{abstract}

\maketitle

\section{Introduction}
Let \(\Omega = \mathcal{C}_0([0,T])\) represent the space of continuous functions on the interval \([0,T]\) that vanish at zero, equipped with the supremum norm, thus making it a Banach space. We consider the Borel \(\sigma\)-field \(\mathcal{F} := \mathcal{B}(\Omega)\) and a probability measure \(P\) under which the canonical process defined by \(B^H_t(\omega) = \omega(t)\) for \(\omega \in \Omega\) constitutes a fractional Brownian motion (fBm) with Hurst parameter \(H > 1/2\). Throughout this paper, we will utilize the probability space \((\Omega, \mathcal{F}, P)\) as the foundation for our analysis.

A fractional Brownian motion \(B^H = \{B^H_t; t \in [0,T]\}\) is a centered Gaussian process with covariance function given by 
\[
R_H(t,s) := \E[B_t^H B_s^H] =  \frac{|t|^{2H} + |s|^{2H} - |t - s|^{2H}}{2}.
\]
Although this process does not exhibit independent increments, it possesses stationary increments. This feature allows us to establish the inequality
\[
\E(|B^H_t - B^H_s|^p) \leq C |t - s|^{pH},
\]
which, by Kolmogorov's continuity criterion, indicates that \(B^H\) has \(\gamma\)-Hölder continuous paths for any \(\gamma \in (0,H)\).

This paper focuses on stochastic delay differential equations (SDDEs) driven by fractional Brownian motion, specifically following the form introduced by Ferrante and Rovira in \cite{ferrante2006stochastic}:
\begin{equation} \label{e: main equation}
\begin{cases}
    X_t = \eta_0 + \int_0^t \sigma(X_{s - r}) \, dB^H_s + \int_0^t b(X_s) \, ds, \quad t \in (0,T], \\
    X_t = \eta_t, \quad t \in [-r,0],
\end{cases}
\end{equation}
where \(\eta\) is a smooth function, \(r > 0\) is a constant delay parameter and $B^H$ is a fBm of Hurst parameter $H > 1/2$. Given the regularity properties of \(B^H\) for \(H > 1/2\), the stochastic integral in this equation can be interpreted as a Riemann-Stieltjes integral, leveraging Young's results in \cite{Young}. Furthermore, Zähle's framework for fractional calculus \cite{Zhle1998IntegrationWR} allows us to represent this integral as a Lebesgue integral via an integration-by-parts formula.

In 2002, Nualart and Rascanu established the existence and uniqueness of solutions for general stochastic differential equations (SDEs) driven by a fBm with Hurst parameter \(H > 1/2\) in their work \cite{rascanu2002differential}. Their proof can be easily adapted to proof the existence and uniqueness of solutions for stochastic delay differential equations of the form \eqref{e: main equation}. In 2006, Ferrante and Rovira presented an alternative approach to proving the existence and uniqueness of solutions to equation \eqref{e: main equation} in \cite{ferrante2006stochastic}. They utilized an inductive strategy that specifically leveraged the delay property, thereby providing a different perspective from that of Nualart and Rascanu. Additionally, Ferrante and Rovira established the existence and smoothness of the solution's density using methods tailored for the delayed framework. A significant advancement in understanding the law of general SDEs was made by Nualart and Saussereau in 2009. In \cite{nualart2009malliavin}, they proved the existence of the density by connecting Malliavin differentiability with Fréchet differentiability. Again, the proof presented for the general framework in \cite{nualart2009malliavin} can also be straightforwardly adapted to SDDEs. In 2012, León and Tindel, proved in \cite{leon:hal-00440655} the smoothness of the density of solutions to SDDEs using rough path techniques, arriving the same conclusion as Rovira and Ferrante in \cite{ferrante2006stochastic} for $H > 1/2$ using a different method.

In this paper, we take as a starting point the results of \cite{ferrante2006stochastic} and \cite{leon:hal-00440655}, that ensure us that the density of the solution to \eqref{e: main equation} is smooth, and we aim to prove that under smoothness conditions on the coefficients this density function is strictly positive in its support.

SDDEs of this nature are significant in various fields, including mathematical finance and biological modeling, where delays in data and external noise—modeled as fractional noise—play a crucial role. For instance, Arriojas et al. \cite{arriojas2006delayed} examined financial models driven by SDDEs, while similar frameworks are utilized in biology to account for delayed feedback influenced by noisy environments.

Our main result, stated formally in Theorem \ref{t: main result}, provides a lower bound for the density function of the solution to the SDDE. Specifically, we prove that under certain regularity and ellipticity conditions on \(\sigma\) and \(b\), the density \(p_t(x)\) of the solution to equation \eqref{e: main equation} satisfies:
\begin{equation} \label{e: bound introduction 1}
    p_t(x) \geq \frac{\E(|X_t-m_t|)}{c_1t^{2H}}\exp\left( -\frac{c_2(x-m_t)^2}{t^{2H}}\right)
\end{equation}
for $t \in (0,r]$ and
\begin{equation} \label{e: bound introduction 2}
p_t(x) \geq \frac{c_3}{t^H} \exp \left( -\frac{c_4(x - \eta_0)^2}{t^{2H}} \right),
\end{equation}
for $t \in (r,T]$, where \(m_t = \E[X_t]\). This result implies, as an immediate corollary, the strict positivity of the density in its support.

The paper is organized as follows: Section \ref{s: preliminaries} introduces the Malliavin calculus tools and techniques used to get to the conclusion of this work. In section \ref{s: lower bound} we prove that the density function of the solution to \eqref{e: main equation} satisfies the bounds \eqref{e: bound introduction 1} and \eqref{e: bound introduction 2}. In the same section, we also state the corollary that concludes the work presented in this paper. Finally, there is an appendix containing two auxiliary results of real analysis that are key in order to derive the Gaussian-type lower bound of the density $p_t(x)$ of the solution to \eqref{e: main equation}.

\section{Preliminaries} \label{s: preliminaries}
\subsection{Wiener space associated to the fractional Brownian motion}

We define $\mathcal{E}$ as the space of step functions of the form
\[
s(t) = \sum_{j=1}^n a_j \1_{[0,t_j]}(t)
\]
with respect to the inner product defined by
\begin{equation} \label{inner product}
    \langle \1_{[0,t]}, \1_{[0,s]} \rangle_{\mathcal{H}} = R_H(t,s)
\end{equation}
where $R_H$ is the covariance function of the fBm. This inner product extends to $\mathcal{E}$ by linearity. We define $\mathcal{H}$ as the closure of $\mathcal{E}$ with respect to the inner product \eqref{inner product}. Moreover, for $\phy, \psi \in \mathcal{H}$, the inner product \eqref{inner product} can be written as
\[
\langle \phy, \psi\rangle_{\mathcal{H}} = \alpha_H \int_0^T  \int_0^T \phy_u \psi_v|u-v|^{2H-2}  du dv,
\]
where
\[
\alpha_H = H(2H-1).
\]
Consider the kernel $K_H : [0,T]^2 \to \R$ defined by
\[
\begin{split}
        &K_H(t,s) = c_H s^{1/2 - H}\int_s^t (u-s)^{H-3/2}u^{H-1/2} du, \\
        &c_H = \sqrt{\frac{H(2H-1)}{\beta(2-2H, H-1/2)}}
\end{split}
\]
for $t > s$ and $K_H(t,s) = 0$ in case $t \leq s$. This kernel induces an isometry between $\mathcal{H}$ and $L^2([0,T])$. Indeed, the operator $K^*_{H}$ defined by
\[
(K_H^* \phy)_s = \int_s^T \phy_t\partial_t K_H(s,t) dt
\]
satisfies that, for every $\phy, \psi \in \mathcal{H}$,
\[
\langle \phy, \psi \rangle_{\mathcal{H}} = \langle K^*_H \phy, K_H^* \psi \rangle_{L^2}.
\]
Moreover, if $W = \{W_t; t \in [0,T]\}$ denotes a standard Brownian motion and we fix $H > 1/2$, then the Volterra process
\[
B^H_t := \int_0^t K_H(t,s) dW_s
\]
is a fractional Brownian motion of Hurst index $H$. Abusing of the notation, we also define $B^H := \{ B^H(h); h \in \mathcal{H}\}$ the isonormal process associated to $\mathcal{H}$. Notice that the process $B^H_t := B^H(\1_{[0,t]})$ is, again, a fractional Brownian motion of Hurst index $H$. We can represent $B^H(h)$ as
\[
B^H(h) = \int_0^T h_t dB^H_t
\]
where this last integral can be understood in the Young sense (see \cite{Young}). 
\subsection{Malliavin Calculus tools}
\begin{defi}
    We define $\mathcal{S}$ as the space of random variables $F$ of the form
    \[
    F = f(B^H(h_1), \dots, B^H(h_n)); \quad f \in \mathcal{C}^{\infty}_p(\R^n), \quad n \geq 1.
    \]
We say that $\mathcal{S}$ is the space of cylindrical random variables.
\end{defi}
We can define the Malliavin derivative in this class of random variables $\mathcal{S}$.
\begin{defi}
    For $F \in \mathcal{S}$, we define $DF$ as the $\mathcal{H}$-valued random variable of the form
    \[
    DF = \sum_{j=1}^n \partial_j f (B^H(h_1), \dots, B^H(h_n))h_j.
    \]
    Given $h \in \mathcal{H}$, we can also define the derivative of $F$ in the direction of $h$ as 
    \[
    D_h F = \langle DF, h\rangle_{\mathcal{H}}.
    \]
    We also define the process $\{D_tF; t \in [0,T]\}$ as
    \[
    D_tF = \sum_{j=1}^n \partial_j f (B^H(h_1), \dots, B^H(h_n))h_j(t).
    \]
\end{defi}
In the same way as we have introduced first order derivatives, we can define derivatives of order $k \geq 2$ as follows: for $t_1, \dots t_k \in [0,T]$, we define $D^{(k)}_{t_1, \dots, t_k}$ as the $\mathcal{H}^{\otimes k}$-valued random variable
\[
D^{(k)}_{t_1, \dots, t_k} F = D_{t_1} \cdots D_{t_k}F. 
\]
For every $k \geq 1$ and $p \geq 1$, the operators $D^{(k)}$ are closable from $\mathcal{S}$ to $L^P(\Omega; \mathcal{H}^{\otimes k}$). We denote by $\D^{k,p}$ the closure of $\mathcal{S}$ with respect to the norm
\[
||F||_{k,p} = \left( \E[|F|^p] + \sum_{j=1}^k \E[||D^{(j)}F||^p_{\mathcal{H}^{\otimes j}}]\right)^{1/p}.
\]
Notice that $\D^{k, p} \subset \D^{l, p}$ if $l\leq k$ and $\D^{k,p} \subset \D^{k,q}$ if $q \leq p$. For a random variable $F \in \D^{1,2}$, we define $\Gamma_F$ and $\Gamma_F^{-1}$ (provided $||DF||^2_{\mathcal{H}} > 0$ a.s.) as
\[
\Gamma_F = ||DF||^2_{\mathcal{H}}, \quad \Gamma_F^{-1} = \left( ||DF||^2_{\mathcal{H}}\right)^{-1}.
\]
In the same way that we have introduced the Malliavin derivative in the classical Malliavin calculus setting, we need to introduce the conditional Malliavin calculus. To this end, let $\mathbb{F} =\{ \mathcal{F}_t; t \in [0,T]\}$ be the natural filtration associated to $B^H$. For a given $t \in [0,T]$ and a random variable $F \in L^1(\Omega)$, we set
\[
E_t[F] = E[F |\mathcal{F}_t].
\]
In the same way as with the classical Malliavin calculus, we denote by $||F||_{k,p,t}$ and $\Gamma_{F,t}$ the following objects
\begin{equation} \label{def conditional quantities BH}
||F||_{k,p,t} = \left( E_t[|F|^p] + \sum_{j=1}^k E_t[||D^{(j)}F||^p_{\mathcal{H}[t,T]^{\otimes j}}]\right)^{1/p}_{,} \quad \Gamma_{F,t} = ||DF||^2_{\mathcal{H}[t,T]},
\end{equation}
for $t \in [0,T]$, where, for $s, t \in [0,T]$ with $s\leq t$ and $\phy \in \mathcal{H}$,
\[
||\phy||^2_{\mathcal{H}[s, t]} = \int_s^t \int_s^t \phy_u \phy_v |u-v|^{2H-2} dudv.
\]
We are now interested in an integration-by-parts formula similar to the one found in \cite{nualart2006malliavin}. For our purposes, we will use a conditional version of this formula. The result is stated in the Wiener process $W$ in $[0,T]$ framework. Since our aim is a conditional IBP formula for the fBm case, the first thing we shall show in order to make sense of this result is that the Malliavin derivative with respect to $W$ and the Malliavin derivative with respect to $B^H$ are related in some way. The relation between these two operators can be stated as follows.

\begin{prop} \label{relation derivatives}
Let $\D^{1,2}_W$ be the Malliavin-Sobolev space associated to $W$. Then, $\D^{1,2} = (K_H^*)^{-1} \D^{1.2}_W$ and for every $F \in \D_W^{1,2}$, we have
\[
D^WF = K_H^* DF
\]
whenever both sides of the equality are well defined.
\end{prop}
\begin{proof}
    See \cite{nualart2006malliavin}.
\end{proof}
In the same way that we have defined the objects \eqref{def conditional quantities BH} in the fractional Brownian motion case, we can define them in the standard Brownian motion case. When working in this framework, we define:
\begin{equation}
||F||_{k,p,t, W} = \left( E_t[|F|^p] + \sum_{j=1}^k E_t[||(D^W)^jF||^p_{L^2[t,T]^{\otimes j}}]\right)^{1/p}, \quad \Gamma_{F,t}^{W} = ||D^W F||^2_{L^2[t,T]}.
\end{equation}
Now that we know that there exists a relation between the derivative operators with respect to $W$ and $B^H$, we can then state the conditional integration by parts formula for $W$ and the result for $B^H$ will be a direct consequence of applying the proposition \ref{relation derivatives}. The result for the Wiener process $W$ is stated in \cite{nualart2006malliavin}. Its conditional version (see, for instance, \cite{kohatsu2003lower} or \cite{besalu2016gaussian}) can be formulated as follows:

\begin{prop} \label{IBP}
    Fix $n \geq 1$. Suppose $F$ is a non-degenerate random variable, let $Z \in \D_{W}^{\infty}$ be an $\mathcal{F}_s$-measurable random variable such that $F +Z$ is a non-degenerate random variable and let $G \in \D_{W}^{\infty}$. We denote by $\1^n:=(1,\dots, 1) \in \R^n$. For any function $g \in \mathcal{C}_p^{\infty}(\R)$ (i.e. $g$ is smooth and $g$, together with all its partial derivatives, have at most polynomial growth) there exists a random variable $H^s_{(\1^n)}(F,G)$ such that
    \begin{equation} \label{e: IBP}
    E_{s}[g^{(n)}(F + Z) G] = E_s[g(F+Z) H^s_{(\1^n)}(F,G)],
    \end{equation}
      where $H^s_{(\1^n)}(F,G)$ is defined recursively by
    \[
    H^s_{(1)}(F,G) = \delta^W_s(G (\Gamma^{-1}_{F,s})DF)
    \]
    and
    \[
    H^s_{(\1^n)}(F,G) = H^s_{(1)}(F, H^s_{(\1^{n-1})}(F,G)).
    \]
    Here, $\delta^W_s$ denotes the Skorohod integral in the interval $[s,T]$. Moreover, for $q_1, q_2, q_3$ such that $\frac{1}{p} = \frac{1}{q_1} + \frac{1}{q_2} + \frac{1}{q_3}$ then the following estimate holds:
    \begin{equation} \label{IBP estimate}
        ||H^s_{(\1^n)}(F,G)||_{p,s, W} \leq c ||(\Gamma^W_{F, s})^{-1}||_{2^{n-1}q_1, s, W}^n||F||^{2(n+1)}_{n+2, 2^n q_2, s, W} ||G||_{n, q_3, s, W}.
    \end{equation}

\end{prop}
Some remarks can be deduced from this proposition.
\begin{rmk} \label{remark 1}
    Notice that using an approximation argument, the conclusion of Proposition \ref{IBP} holds for $g(F+Z) = \1_{\{F+Z > x\}}$. Indeed, consider a sequence $\{g_k; k \geq 1\}$ of smooth compactly supported functions that converge to de Dirac's delta centered at a point $x \in \R$. Then, using the dominated convergence theorem, it is clear that the right hand side of \eqref{IBP} converges to
    \[
    E_s[\1_{\{F+Z > x\}}H_{(\1^n)}^s(F,G)],
    \]
    while, concerning the left hand side, the limit is
    \[
    E_s[g^{(n)}(F+Z)G]
    \]
    where $g^{(n)}$ is the $n$-th order distributional derivative of $g$.
\end{rmk}

\begin{rmk} \label{remark 2}
    It is interesting to study (and will be useful in the future) how the integration by parts formula changes with reescalings of $F$. More precisely, given $\mu \in \R$, we want to see how does the estimate \eqref{IBP estimate} change when we consider $\mu F$ instead of $F$. We will prove by using induction on $n$ that
    \begin{equation} \label{IBP reescaling}
        H^s_{(\1^{n})}(\mu F, G) = \frac{1}{\mu^n}H^s_{(\1^n)}(F,G).
    \end{equation}
    First we cover the case $n = 1$. Observe that
    \[
    (\Gamma^{W}_{\mu F,s})^{-1} = (||\mu D^W F||^2_{L^2[s,T]})^{-1} = \mu^{-2} (\Gamma^W_{F,s})^{-1}
    \]
    and obviously, $D^W\mu F = \mu D^W F$. Using then the linearity of $\delta^W_s$ we have
    \[
    H^s_{(1)}(\mu F, G) = \frac{1}{\mu} H^s_{(1)}(F,G).
    \]
    Observe also that $H^s_{(1)}(F, \mu G) = \mu H^s_{(1)}(F, G)$. Assume that it holds for $n-1$. For $n$ we have
    \[
    \begin{split}
        H^s_{(\1^{n}}(\mu F, G) = &H_{(1)}(\mu F, H^s_{(\1^{n-1})}(\mu F, G)) \\
        = & \frac{1}{\mu}H_{(1)}\left( F, \frac{1}{\mu^{n-1}} H^s_{(\1^{n-1})}(F, G)\right) \\
        = & \frac{1}{\mu^n}H_{(1)}\left( F, H^s_{(\1^{n-1})}(F, G)\right) \\
        = & \frac{1}{\mu^n}H^s_{(\1^n)}(F,G).
    \end{split}
    \]
    This proves the desired scaling property of $H^s_{(\1^n)}$.
\end{rmk}
Notice that from the relation given by proposition \ref{relation derivatives}, the formula still holds for the Malliavin spaces associated to the fractional Brownian motion with the change of the underlying Hilbert space and shifting the random variables via $K_H^*$. Hence, Proposition \ref{IBP} also holds in the fractional Brownian motion framework.

\section{Lower bound for the density} \label{s: lower bound}
The objective of this section is to get all the tools needed in order to give a proof of the lower bound. The main result is encapsulated in the following Theorem:
\begin{teo} \label{t: main result}
    Let $X_t$ be the solution to equation \eqref{e: main equation} with $\sigma, b \in \mathcal{C}^{\infty}_b(\R)$, $\eta \in \mathcal{C}^{\infty}_b((-r,0))$, and moreover, there exist two constants $0 < \lambda < \Lambda$ such that $\lambda \leq ||\sigma||_{\infty} \leq \Lambda$. Then, the density function $p_t(x)$ of $X_t$ satisfies
    \begin{equation} \label{e: bound main result 0r}
        p_t(x) \geq \frac{\E(|X_t-m_t|)}{c_1 t^{2H}}\exp \left( - \frac{c_2(x-m_t)^2}{t^{2H}}\right)
    \end{equation}
    if $t \in (0,r]$, and
    \begin{equation} \label{e: bound main result rT}
        p_t(x) \geq \frac{c_3}{t^H}\exp \left( - \frac{c_4(x-\eta_0)^2}{t^{2H}}\right)
    \end{equation}
    if $t \in (r,T]$, where $c_1, c_2, c_3, c_4 > 0$ are real constants and $m_t = \E[X_t]$.
\end{teo}

As mentioned, in order to prove this result, we will make a distinction in the cases $(0,r]$ and $(r, T]$. For the case $(0,r]$, we will make use of the techniques developed in \cite{NourdinViens}, as it is the most comfortable method in order to study the density of the solution $X_t$ when $t \in (0,r]$ because of the structure of the equation in this interval of time.
\subsection{The case $t \in (0,r]$} \label{First case}

\subsubsection{A general bounding technique}
In order to illustrate how to obtain a bound for the density in this case, we will briefly recall the method developed by Nourdin and Viens in \cite{NourdinViens}. The bound relies on the following results:

    \begin{teo} \label{g_F}
        Let $F \in \D^{1,2}$ with zero mean, let $g_F(x)$ be the function defined as
        \[
        g_F(x) = E\left[\langle DF, -DL^{-1}F\rangle_{\mathcal{H}} | F = x\right],
        \]
        where $L$ denotes the Ornstein-Uhlenbeck operator associated to $B^H$. The law of $F$ has a density $p_F(x)$ if and only  if the random variable $g_F(F)$ is strictly positive almost surely. In this case, the support of $p_F$, $supp(p_F$), is a closed interval of $\R$ containing zero and, for almost $x \in supp(p_F)$,
        \[
        p_F(x) = \frac{\E[|F|]}{2g_F(x)} \exp \left(-\int_0^x \frac{z}{g_F(z)}dz\right).
        \]
    \end{teo}
This result has the following consequence.
\begin{corol} \label{corol g_F}
    If there exist $\sigma_{min}, \sigma_{max} > 0$ such that
    \[
    \sigma_{min}^2 \leq g_F(F) \leq \sigma^2_{max},
    \]
    then $F$ has  density function satisfying
    \[
    \frac{\E[|F|]}{2\sigma^2_{max}}\exp\left( -\frac{x^2}{2\sigma^2_{min}} \right) \leq p_F(x) \leq \frac{\E[|F|]}{2\sigma^2_{min}}\exp\left( -\frac{x^2}{2\sigma^2_{max}} \right).
    \]
\end{corol}
Even though this result is directly giving us the lower and the upper bound, it is not clear from the definition of $g_F$ how to analyze this function. Indeed, we have to keep in mind that in our case, $F$ will be related to the solution of an SDDE, so computing $DL^{-1}F$ seems, a priori, a difficult task. Concerning our case, let $B^H$ be a fractional Brownian motion with Hurst parameter $H> 1/2$. As an abuse of notation, we will denote by $B^H$ the isonormal process asociated to $\mathcal{H}$. A way of computing $g_F(x)$ is the following (see \cite{NourdinViens}):

\begin{prop}
    Assume $DF = \Phi_F(B^H)$ for a measurable function $\Phi_{F} : \R^{\mathcal{H}} \to \mathcal{H}$. Then, we have
    \[
    \langle DF, -DL^{-1}F \rangle_{\mathcal{H}} = \int_{0}^{\infty} e^{-\theta} \langle \Phi_F(B^H), \E' \left(\Phi_F(e^{-\theta}B^H + \sqrt{1-e^{-2\theta}}B^{H'}) \right)\rangle_{\mathcal{H}} d\theta
    \]
    and, therefore,
    \[
    g_F(F) = \int_{0}^{\infty} e^{-\theta} \mathbf{E}\left(\langle \Phi_Z(B^H), \Phi_Z(e^{-\theta}B^H + \sqrt{1-e^{-2\theta}}B^{H'}) \rangle_{\mathcal{H}}|F\right) d\theta,
    \]
    where $B^{H'}$ stands for an independent copy of $B^H$ such that $B^H$ and $B^{H'}$ are defined in the probability space $(\Omega \times \Omega', \mathcal{F} \otimes \mathcal{F}', P \times P')$, $\E'$ denotes the expectation with respect to $P'$ and $\mathbf{E}$ denotes the expectation with respect to $P \times P'$.
\end{prop}
We will rely heavily on this way of computing $g_F$ in order to prove the following bounds for $t\in (0,r]$:
\begin{prop} \label{density in (0,r]}
    Let $X$ be the solution \eqref{e: main equation} with $\sigma, b$ satisfying the hypothesis of the Theorem \ref{t: main result}. If $\sigma$ is $(\lambda, \Lambda)$-elliptic, then, for every $t \in (0,r]$, there exist constants $c_1 < c_2$ such that the density $p_t(x)$ of $X_t$ satisfies:
    \[
    \begin{split}
\frac{c_1\E[|X_t-m_t|]}{2\Lambda^2 t^{2H}}\exp\left( -\frac{(x-m_t)^2}{2\lambda^2t^{2H}}\right) \leq &p_t(x) \\
\leq &\frac{c_2\E[|X_t - m_t|]}{2\lambda^2 t^{2H}}\exp\left( -\frac{(x-m_t)^2}{2\Lambda^2 t^{2H}}\right),
\end{split}
    \]
    where $m_t = \E[X_t]$.
\end{prop}

\begin{proof}
    We will apply the method recently introduced. Notice that for $t \in (0,r]$, $X_t$ solves the equation
    \[
    X_t = \eta_0 + \int_0^t \sigma(\eta(s-r))dB^H_s + \int_0^t b(X_s) ds.
    \]
    Hence, the Malliavin derivative of $X_t$ in the direction $s < t$ satisfies the following equation:
    \[
    D_sX_t =  \sigma(\eta(s-r)) + \int_0^t \sigma'(\eta(u-r)) D_s\eta(u-r) dB^H_u + \int_0^t b'(X_u) D_sX_u du.
    \]
    Since $\eta$ is deterministic, $D_s\eta(u-r) = 0$, so the equation satisfied by $D_sX_t$ is reduced to
    \[
    D_sX_t = \sigma(\eta(s-r)) + \int_0^t b'(X_u)D_sX_u du.
    \]
    This is an ODE with initial condition $\sigma(\eta(s-r))$, so the explicit solution to this equation is
    \[
    D_sX_t = \sigma(\eta(s-r)) \exp\left( \int_s^t b'(X_u) du \right).
    \]
    Moreover, from the fact that $|b'(X_u)| \leq ||b'||_{\infty}$ and the ellipticity condition on $\sigma$ we deduce that there exists $M > 0$ such that
    \begin{equation} \label{der0r}
    \lambda e^{-Mr} \leq D_sX_t \leq \Lambda e^{Mr}.
    \end{equation}
    The important conclusion of this bound is that \eqref{der0r} holds uniformly with respect to $\omega \in \Omega$. Notice that one has that $X_t \in \D^{1,2}$, but the method works for centered random variables. In order to apply the bounding technique, we will apply the method to $F = X_t - \E(X_t)$. Since $\E(X_t)$ is a real number, $DF = DX_t$. Moreover, recall that if $\phy, \psi \in \mathcal{H}$, then
    \[
    \langle \phy, \psi\rangle_{\mathcal{H}} = \int_0^T \int_0^T \phy_u\psi_v|u-v|^{2H-2} du dv.
    \]
    Hence, if we define $F^{\theta} := F(e^{-\theta}\omega + \sqrt{1-e^{-2\theta}}\omega')$ we can write $g_F$ as
    \[
    g_F(F) = \int_{0}^{\infty} e^{-\theta} \E \left[\E' \left(\int_0^t \int_0^t D_uX_t (D_vX_t)^{\theta} |u-v|^{2H-2} du dv | F \right)\right] d\theta
    \]
    using the bounds \eqref{der0r} we can easily find that there exist $c_1, c_2 > 0$ such that
    \[
   c_1 \lambda^2 t^{2H} \leq g_F(F) \leq c_2 \Lambda^2 t^{2H}
    \]
    which finishes the proof.
\end{proof}
 Notice that renaming the constants, we derive the bound \ref{e: bound main result 0r}. Some comments about this method is that, as the reader can observe, it is extremely comfortable to study Stochastic Differential Equations driven by an additive noise and, under suitable hypothesis on $\sigma$, the same arguments work for equations of the type \eqref{e: main equation} for $t \in (0,r]$. The fact that for $t > r$ the difussion coefficient is random makes it impossible to keep applying this same method.

\subsection{The case $t > r$} \label{second case}
As mentioned, a natural approach to face the case $t > r$ is continuing with the same approach as in the case $t \in (0,r]$. However, the structure of the equation forces us to look for another approach. 

The strategy in order to get the lower bound for $p_t(x)$ is inspired in the strategy found in \cite{kohatsu2003lower} and adapted to our context. In \cite{besalu2016gaussian} can be found an application of the method to general equations driven by a fractional Brownian motion with $H > 1/2$. One big difference between this method and the one used for the case $t \in (0,r]$ is that the latter method produces as well an upper bound for $p_t(x)$. The method that we will use will only give us a lower bound, but it can be complemented by an upper bound using similar arguments as in \cite{baudoin:hal-00931118} but adapted to the delay setting.

The proof of the lower bound we give in this paper is long and technical. One of the main objectives of this work is illustrate how the delays are actually helpful in order to get simpler proofs. 

The first step in order to work with this method is representing the density function $p_t(x)$ as $\E[\delta_x(X_t)]$. To do so, we rely on the following result.
\begin{teo} \label{representation}
    Under the hypothesis of Theorem \ref{t: main result}, the unique solution to \eqref{e: main equation} is a non-degenerate random variable in the sense of Malliavin for all $t \in (0,T]$, that is,
    \begin{itemize}
        \item [(i)] $X_t \in \D^{\infty}$.
        \item [(ii)] $\Gamma^{-1}_{X_t}  > 0$ almost surely, and $\Gamma^{-1}_{X_t}\in \bigcap_{p\geq 1}L^p(\Omega)$.
    \end{itemize}
    Moreover, the density $p_t(x)$ of $X_t$ admits the representation $p_t(x) = \E[\delta_x(X_t)]$ where $\delta_x$ denotes the Dirac's delta measure at $x$.
\end{teo}
\begin{proof}
Items $(i)$ and $(ii)$ are proved in \cite{ferrante2006stochastic}. Hence, by the same argument as in \cite{kohatsu2003lower} we obtain that the density $p_t(x)$ of $X_t$ can be expressed as
\[
p_t(x) = \E[\delta_{x}(X_t)],
\]
as desired.
\end{proof}

In order to analyze $p_t(x) = \E[\delta_x(X_t)]$, we will construct an approximating sequence $\{F_n; 0 \leq n \leq N\}$ such that $F_N = X_t$ and we will evaluate $p_t(x)$ via evaluating the conditional densities of every $F_n$.  By construction, analyzing $p_t(x)$ is equivalent to analyzing $\E[\delta_x(F_N)]$ thanks to the representation given by Theorem \ref{representation}. In order to construct such an approximating sequence, we will construct a partition $\pi = \{0 = t_0 < t_1 < \cdots < t_N = t\}$ such that $|t_n - t_{n-1}| < \eps$ for  a small enough $\eps > 0$ that will be chosen conveniently in the future. Then, $F_n$ will be an $\mathcal{F}_{t_n}$-measurable random variable. In order to exploit the property of the delay, we will consider $F_{n}$ of the form
\[
F_n = F_{n-1} + I_n + R_n
\]
where $I_n$ is a stochastic integral of a predictable process and $R_n$ is the remainder term. In the case of a general Stochastic Differential Equation, the choice of $I_n$ is not direct. In fact, choosing $I_n$ as the integral of a predictable process makes the term $R_n$ to be difficult to study. Thanks to the delays, our choice of $I_n$ is straightforward and natural. We therefore define
\[
F_{n-1} = \eta_0 + \int_0^{t_{n-1}} \sigma(X_{s-r})dB^H_{s} + \int_0^{t_{n-1}}b(X_s)ds,
\]
\[
I_n = \int_{t_{n-1}}^{t_n}  \sigma(X_{s-r})dB^H_s,
\]
and
\[
R_n = \int_{t_{n-1}}^{t_n} b(X_s) ds.
\]
This allows us to write
\[
p_t(x) = \E[\delta_x(F_N)] = \E[\delta_x(F_{N-1} + I_{N} + R_N)]
\]
which, by the properties of the conditional expectation, this last expression can also be written as
\[
p_t(x) = \E[ E_{t_{N-1}}(\delta_x(F_{N-1} + I_{N} + R_N))].
\]
Therefore, in order to get information about $p_t(x)$ we first need to get as much information as we can about $E_{t_{N-1}}[\delta_x(F_{N-1} + I_{N} + R_N)]$. Using a Taylor expansion, this term can be written as
\[
\begin{split}
E_{t_{N-1}}[\delta_x(F_{N-1} + I_{N} + R_N)] = & E_{t_{N-1}}[\delta_x(F_{N-1} + I_N)]  \\
+ &E_{t_{N-1}}\left[\int_0^1 \frac{d}{d\rho}\delta_x(F_{N-1} + I_N + \rho R_N) R_n d\rho \right]
\end{split}
\]
where the derivative in the third term has to be understood as the second order derivative in the distributional sense of $\1_{\{F_{N-1} + I_N + \rho R_N > x\}}$. Inspired by this decomposition, we write for all $1 \leq n \leq N$,

\begin{equation} \label{J}
J_{1,n} = E_{t_{n-1}}[\delta_x(F_{n-1} + I_n)], \quad J_{2,n} = E_{t_{n-1}}\left[\int_0^1 \frac{d}{d\rho}\delta_x(F_{n-1} + I_n + \rho R_n) R_n d\rho \right].
\end{equation}
So a first (and natural) lower bound is
\[
E_{t_{n-1}}[\delta_x(F_{n-1}+I_n + R_n)] = J_{1,n} + J_{2,n} \geq J_{1.n} - |J_{2,n}|.
\]
Hence, our most immediate objective is to give a lower bound for $J_{1,n}$ and an upper bound for $|J_{2,n}|$.
\subsubsection{Lower bound for $J_{1,n}$} \label{bound J1n}
The key in order to give a lower bound for $J_{1,n}$ is the following result:

\begin{prop} \label{fita inf J_1}
    Let $J_{1,n}$ be defined as \eqref{J}. Then, if the partition $\pi$ is such that $\sup_{0 \leq n \leq N}|t_{n+1} - t_n |< \eps < r$ and $|t_n - t_{n-1}|^{2H} = \sigma^2_N := \alpha_H\frac{t^{2H}}{N} $, then
    \[
    J_{1,n}\geq \frac{1}{\sqrt{2\pi \Lambda^2 \sigma^2_N}} \exp \left( - \frac{(x-F_{n-1})^2}{2 \lambda^2 \sigma^2_N} \right).
    \]
\end{prop}

\begin{proof}
    Notice that the term $J_{1,n}$ is the conditional density of $F_{n-1} + I_n$ with respect to $\mathcal{F}_{t_n}$. Hence, in order to get a lower bound for $J_{1,n}$, we need to find and bound this conditional density.
    
    On the one hand, $F_{n-1}$ is $\mathcal{F}_{t_{n-1}}$-measurable, so it behaves like a constant when we condition it with respect to $\mathcal{F}_{t_{n-1}}$. From the fact that $t_{n} - t_{n-1}  < r$, we have that $\sigma(X_{s-r})$ is $\mathcal{F}_{t_{n-1}}$ measurable for all $s \in [t_{n-1}, t_n]$ and therefore, the law of $F_{n-1} + I_n$, when conditioned to $\mathcal{F}_{t_{n-1}}$, follows a Gaussian distribution with mean $F_{n-1}$ and variance $||\sigma(X_{\cdot - r})||^2_{\mathcal{H}[t_{n-1}, t_n]}$. Moreover, from the ellipticity condition on $\sigma$ and the definition of $\sigma^2_N$ we readily find that
    \[
    \lambda^2 \sigma^2_N \leq ||\sigma(X_{\cdot - r})||^2_{\mathcal{H}[t_{n-1}, t_n]} \leq \Lambda^2 \sigma^2_N.
    \]
    This allows us to conclude that
    \[
    J_{1,n}\geq \frac{1}{\sqrt{2\pi \Lambda^2 \sigma^2_N}} \exp \left( - \frac{(x-F_{n-1})^2}{2 \lambda^2 \sigma^2_N} \right),
    \]
    as wanted.
\end{proof}

\begin{rmk}
    It is still left to check the fact that such a partition defined as in Proposition \ref{bound J1n} exists. However, proving that there exists a unique partition $\pi = \{0 = t_0 < t_1 < \cdots < t_N = t\}$ of the interval $[0,t]$ with the property that $|t_n - t_{n-1}|^{2H} = \sigma^2_N := \alpha_H \frac{t^{2H}}{N}$ is direct.
\end{rmk}
We will see in the future that the parameter $N$ can be chosen arbitrarily large, so we can construct the partition $\pi$ such that $\sup_{1\leq n \leq N}|t_n - t_{n-1}| < \eps < r$, that is what is needed in order to prove Proposition \ref{bound J1n}.

\subsubsection{Upper bound for $|J_{2,n}|$} \label{bound J2n}
The result concerning the bound for $|J_{2,n}|$ is the follwing.
\begin{prop} \label{bound J_2n a priori}
    Let $J_{2,n}$ be defined as in \eqref{J}. There exist two constants $C, \gamma > 0$ such that
    \begin{equation} \label{e: bound J_2n a priori}
    |J_{2,n}| \leq C\frac{N^{-\gamma}}{\sigma_N}.
    \end{equation}
\end{prop}
The proof of this proposition is not immediate. In order to conclude \eqref{e: bound J_2n a priori}, we need to derive a first estimate using the integration by parts formula and then we will bound each of the terms involved using some technical lemmas. First, recall that
\[
J_{2,n} = E_{t_{n-1}}\left[\int_0^1 \frac{d}{d\rho}\delta_x(F_{n-1} + I_n + \rho R_n) R_n d\rho \right].
\]
We introduce a new random variable $U_n$ defined by 
\begin{equation} \label{definition Un}
    \sigma_N U_n  = I_n + \rho R_n.
\end{equation} 
With this random variable, $J_{2,n}$ can now be written as
\[
J_{2,n} = E_{t_{n-1}}\left[ \int_0^1 \frac{d}{d\rho}\delta_x (F_{n-1} + \sigma_N U_n)R_n d\rho \right].
\]
Using Fubini's theorem, we can rewrite $J_{2,n}$ as
\[
J_{2,n} = \int_0^1 E_{t_{n-1}} \left[ \frac{d}{d\rho}\delta_x(F_{n-1} + \sigma_N U_n) R_n  \right]d\rho
\]
Integration by parts formula \ref{IBP} for second order derivatives and Remark \ref{remark 1} now yields
\[
J_{2,n} =  \int_0^1 E_{t_{n-1}} \left[ \1_{\{I_n + \rho R_n > x- F_{n-1}}\} H^{t_{n-1}}_{(\1^2)}( \sigma_N U_n, R_n)\right] d\rho.
\]
Using then Remark \ref{remark 2} , this last expression can be expressed as
\[
J_{2,n} = \sigma_N^{-2} \int_0^1 E_{t_{n-1}} \left[ \1_{\{I_n + \rho R_n > x- F_{n-1}}\} H^{t_{n-1}}_{(\1^2)}( U_n, R_n)\right] d\rho.
\]
Now, from the fact that $\1_{\{I_n + \rho R_n > x-F_{n-1}\}} \leq 1$, Hölder's inequality and the estimate \eqref{IBP estimate}, $J_{2,n}$ be bounded in the following way:
\[
 |J_{2,n}| \leq c \sigma_N^{-2} A_1 \int_0^1 A_2(\rho)A_3(\rho) d\rho,
\]
where
\[
    A_1 = ||R_n||_{k_1, p_1, t_{n-1}}, 
\]
\[
 A_2(\rho) = || \Gamma^{-1}_{U_n, t_{n-1}}||_{p_3, t_{n-1}}^{k_3},
\]
and
\[
A_3(\rho) = ||U_n ||^{k_4}_{k_2, p_2, t_{n-1}}
\]
for some constants $k_1, k_2, k_3, k_4, p_1, p_2, p_3 > 0$. The exact value of the constants is not extremely important, but they can be determined according to the integration by parts formula \ref{IBP}. Indeed, for the second order derivative case, we know that $k_1 = 2$, $k_2 = 4$, $k_3 = 2$, $k_4 = 6$ and $p_1 = q_3$, $p_2 = 4q_2$, $p_3 = 2q_1$ where $q_1, q_2, q_3$ satisfy $\frac{1}{q_1} + \frac{1}{q_2} + \frac{1}{q_3} = 1$. The bound on the quantities $A_j$ relies on the following lemma:
\begin{lema} \label{X uniformly bounded}
    Let $\pi$ be the partition defined as in Proposition \ref{fita inf J_1}. If $s_1, \dots, s_j, \tau \in [t_{n-1},t_n]$ with $s_1 \vee \dots \vee s_j <\tau$, then $D^{(j)}_{s_1, \dots, s_j} X_{\tau}$ is uniformly bounded in $\omega \in \Omega$.
\end{lema}
\begin{proof}
    The result for the first order derivative is straight-forward. Indeed, differentiating equation \eqref{e: main equation} in the direction $s \in [t_{n-1}, t_n]$ we obtain
    \[
    D_sX_{\tau} = \sigma(X_{s-r}) + \int_{0}^{\tau} D_sX_{u-r}\sigma'(X_{u-r})dB^H_u + \int_0^{\tau} b'(X_u)D_sX_u du.
    \]
    Now, since $D_sX_{u-r} = 0$ for all $u \in [0,\tau]$ due to the choice of the partition, last expression can be written as
    \[
    D_sX_{\tau} = \sigma(X_{s-r}) + \int_0^{\tau} b'(X_u) D_sX_u du.
    \]
    This equation has an explicit solution
    \[
    D_sX_{\tau} = \sigma(X_{s-r}) \exp\left(\int_s^{\tau} b'(X_u) du \right).
    \]
    Finally, from the hypothesis on $\sigma$ and $b'$, we have that there exists a constant $C  > 0$ such that
    \[
    |D_sX_{\tau}| \leq C.
    \]
    In order to prove the result for higher order derivatives, we will use an induction argument. Assume that all derivatives up to order $j-1$ are bounded. The derivative of order $j$ satisfies the following equation:
    \begin{equation} \label{HOD}
    \begin{split}
        D^{(j)}_{s_1, \dots , s_j} X_{\tau} = &\sum_{q = 1}^j D_{s_1} \cdots \check{D_{s_q}} \cdots D_{s_j}\sigma(X_{s-r}) \\
         + &\int_{0}^{\tau}  D_{s_1} \dots D_{s_j} \sigma(X_{u-r}) dB^H_u \\
        +  &\int_{0}^{\tau} D_{s_1} \dots D_{s_j} b(X_{u}) du.
    \end{split}
    \end{equation}
    Now, since $s_1\vee \cdots \vee s_j - r < t_{n-1}$ and $\tau \in [t_{n-1}, t_n]$, then equation \eqref{HOD} can be written as
    \begin{equation} \label{e: induction}
        D^{(j)}_{s_1,\dots, s_j} X_{\tau} = \sum_{q = 1}^j D_{s_1} \cdots \check{D_{s_q}} \cdots D_{s_j}\sigma(X_{s_{q}-r})
        + \int_{s_1 \vee \cdots \vee s_j}^{\tau} D_{s_1} \dots D_{s_j} b(X_{u}) du.
    \end{equation}
   On the one hand, each term of the form
   \[
   D_{s_1} \cdots \check{D_{s_q}} \cdots D_{s_j}\sigma(X_{s_{q}-r})
   \]
   involves involves derivatives of orders $k = 1, \dots, j-1$ of $X_t$ and derivatives of orders $l = 1, \dots j$ of $\sigma$. Since $\sigma$ has bounded derivatives of all orders and, by induction hypothesis, all derivatives up to order $j-1$ of $X_t$ are bounded we get that there exists $C_1 > 0$ such that
   \[
   \sum_{q = 1}^{j} |D_{s_1} \cdots \check{D_{s_q}} \cdots D_{s_j}\sigma(X_{s_{q}-r})| \leq C_1.
   \]
   Concerning the term $D_{s_1} \dots D_{s_j} b(X_{u})$, we use the product rule and the chain rule for Malliavin derivatives $j$ times and we find that
   \[
   D_{s_1} \dots D_{s_j} b(X_{u}) = \sum_{k=1}^j b^{(k)}(X_u) \sum_{\Pi \in \mathcal{P}(j,k)} \prod_{l=1}^kD^{(|\Pi_l|)}_{s_{\Pi_l}}X_u
   \]
   where $\mathcal{P}(j,k)$ denotes the set of all partitions of $\{s_1, \dots, s_j\}$ into $k$ subsets, $\Pi_l$ denotes the $l$-th subset of this partition. $|\Pi_l|$ denotes its length and $s_{\Pi_l} = (s_{i_1}, \dots, s_{i_{|\Pi_l}})$ if $\Pi_l = \{s_{i_1}, \dots, s_{i_{|\Pi_l|}}\}$. Even though it is a sum with a lot of terms, the only one that involves a $j$-th order derivative of $X_u$ is
   \[
   b'(X_u) D^{(j)}_{s_1, \dots, s_j}X_u.
   \]
   Using that all derivatives of $b$ are bounded and all derivatives of $X_u$ up to order $j-1$ are bounded, we find that there exist $C_2, C_3 > 0$ such that
   \begin{equation} \label{higher order derivatives of b}
    D_{s_1} \dots D_{s_j} b(X_{u}) \leq C_2 + C_3 D^{(j)}_{s_1, \dots, s_j}X_u.
   \end{equation}
   Plugging this into \eqref{e: induction}, we are lead to
   \[
   D^{(j)}_{s_1, \dots, s_j}X_{\tau} \leq C_1 + \int_0^{\tau} \left( C_2 + C_3  D^{(j)}_{s_1, \dots, s_j}X_u \right) du.
   \]
   Using Gronwall's lemma, we conclude that $|D^{(j)}_{s_1, \dots, s_j}X_{\tau}|$ is uniformly bounded.
\end{proof}
We can now proceed to derive the estimates. Concerning the following 3 lemmas, in the proofs we will use $C$ as a universal constant that may switch from line to line. If several different constants are involved, we will name them $C_1, C_2, \dots$. We drop the dependency of the constant on the parameters for sake of streamlining the reading of the paper.

\begin{lema}[Bound for $A_1$] \label{bound for A_1} Recall that $A_1 = ||R_n||_{k_1, p_1, t_{n-1}}$. Under the same hypothesis as in Lemma \ref{X uniformly bounded}, there exists a constant $C > 0$ such that
\[
 A_1 \leq C N^{-1/2H}.
\]
In particular, there exists $\gamma > 0$ and $C > 0$ such that
\[
\sigma_N^{-2}A_1 \leq C \frac{N^{-\gamma}}{\sigma_N}.
\]
\end{lema}
\begin{proof}
By definition,
    \[
    A_1^{p_1} = ||R_n||^{p_1}_{k_1, p_1, t_{n-1}} = E_{t_{n-1}}(|R_n|^{p_1}) + \sum_{j=1}^{k_1} E_{t_{n-1}}\left(||D^{(j)}R_n||^{p_1}_{\mathcal{H}^{\otimes j}[t_{n-1},t_n]}\right).
    \]
    The strategy of the proof is checking that $|R_n|$ is of the order of $N^{-1/2H}$, and the Malliavin derivatives are $o(N^{-1})$, so that the asymptotic behaviour of $A_1$ is dominated by the first term of $A_1$, and the terms involving the derivatives of $R_n$ are negligible. For the first term, notice that $|R_n| \leq \int_{t_{n-1}}^{t_n}| b(X_s)| ds \leq ||b||_{L^{\infty}}|t_n-t_{n-1}|$. Since the partition is chosen so that $|t_n - t_{n-1}|^{2H} =\alpha_H \frac{t^{2H}}{N}$ , we can conclude that there exists a constant $C > 0$ such that $|R_n| \leq C N^{-1/2H}$, so we can find $C > 0$ such that $E_ {t_{n-1}}(|R_n|^{p_1}) \leq C N^{-p_1/2H}$, which is a consistent estimate with what we want to prove.
    
    The rest of the proof is done by bounding all the derivatives of order $j$ from $1$ up to $k_1$. We will show the bound for the first derivative since its computation will be useful in future lemmas, and the conclusion for the higher order derivatives can be easily derived by using the estimate \eqref{higher order derivatives of b} and the fact that the Malliavin derivatives of $X_t$ are bounded thanks to Lemma \ref{X uniformly bounded}. For the first derivative,
    \[
    ||DR_n||^{p_1}_{\mathcal{H}[t_{n-1}, t_n]} = \left(\int_{t_{n-1}}^{t_n} \int_{t_{n-1}}^{t_n} D_uR_nD_vR_n|u-v|^{2H-2}du dv\right )^{p_1/2}.
    \]
    Now, using the fact that
    \[
    |D_sR_n| \leq \int_{t_{n-1}}^{t_n} |b'(X_u) D_sX_u| du \leq C|t_n - t_{n-1}| \leq CN^{-1/2H}
    \]
    we find that
    \begin{equation} \label{bound norm of DR_n}
        ||DR_n||^{p_1}_{\mathcal{H}[t_{n-1}, t_n]} \leq C \left( N^{-1/H} |t_n - t_{n-1}|^{2H}\right)^{p_1/2} \leq CN^{-\frac{p_1(1+\frac{1}{H})}{2}}.
    \end{equation}
    Hence,
    \[
    E_{t_{n-1}}\left( ||DR_n||^{p_1}_{\mathcal{H}[t_{n-1}, t_n]}\right) \leq CN^{-\frac{p_1(1+\frac{1}{H})}{2}}
    \]
    as desired.
\end{proof}

\begin{lema}[bound for $A_2(\rho)$] \label{bound for A2} Recall that $A_2(\rho) = ||\Gamma^{-1}_{U_n, t_{n-1}}||^{k_3}_{p_3, t_{n-1}}$. Under the same hypothesis as in Lemma \ref{X uniformly bounded}, the quantity $A_2(\rho)$ is uniformly bounded.
\end{lema}
\begin{proof}
    Recall that
    \[
    \sigma_N U_n = I_n + \rho R_n,
    \]
    so 
    \[
    \sigma^2_N ||DU_n||^2_{\mathcal{H}[t_{n-1}, t_n]} = ||DI_n + \rho DR_n||^2_{\mathcal{H}[t_{n-1}, t_n]}.
    \]
    Now, using Lemma \ref{Ap 1} in the appendix and the fact that $-\rho \geq -1$ then we have
    \[
    \sigma^2_N ||DU_n||^2_{\mathcal{H}[t_{n-1}, t_n]} \geq \frac{||DI_n||^2_{\mathcal{H}[t_{n-1}, t_n]}}{2} - ||DR_n||^2_{\mathcal{H}[t_{n-1}, t_n]}.
    \]
    Since
    \[
    I_n = \int_{t_{n-1}}^{t_n} \sigma(X_{s-r})dB^H_s,
    \]
    then it is clear that for $s \in [t_{n-1}, t_n]$ we have
    \[
    D_sI_n = \sigma(X_{s-r}).
    \]
    From the fact that $\sigma(X_{s-r}) > \lambda$ and $\int_{t_{n-1}}^{t_n}\int_{t_{n-1}}^{t_n} |u-v|^{2H-2} dudv = C \sigma^2_N$ (see Lemma \ref{Ap 2} in the appendix) we derive the following estimate:
    \[
    ||DI_n||^2_{\mathcal{H}[t_{n-1},t_n]} = \int_{t_{n-1}}^{t_n} \int_{t_{n-1}}^{t_n} \sigma(X_{u-r})\sigma(X_{v-r}) \cdot |u-v|^{2H-2} du dv \geq C \sigma^2_N.
    \]
    Moreover, by the inequality \eqref{bound norm of DR_n} in the proof of the previous lemma for $p_1 = 2$, we have that $||DR_n||^2_{\mathcal{H}[t_{n-1}, t_n]} \leq CN^{-1-\frac{1}{H}}$. Since $\sigma^2_N = O(N^{-1})$, for $N$ big enough we deduce that there exists $C_1, C_2, C_3 > 0$ such that
    \[
    \sigma^2_N ||DU_n||^2_{\mathcal{H}[t_{n-1},t_n]} \geq C_1\sigma^2_N - C_2 N^{-1-\frac{1}{H}} \geq C_3 \sigma^2_N
    \]
    so
    \[
    ||DU_n||_{\mathcal{H}[t_{n-1}, t_n]}^2 \geq C_3 \frac{\sigma^2_N}{\sigma^2_N} = C_3 > 0.
    \]
    This implies that
    \[
    \Gamma^{-1}_{U_n, t_{n-1}} \leq \frac{1}{C_3}.
    \]
    Taking conditional norms we deduce that there exists $C > 0$ satisfying $||\Gamma^{-1}_{U_n, t_{n-1}}||^{k_3}_{p_3, t_{n-1}} \leq C$, which proves the uniform bound of $A_2(\rho)$.
\end{proof}

\begin{lema}[bound for $A_3(\rho)$] \label{bound for A3} Recall that $A_3(\rho) = ||U_n||^{k_4}_{k_2, p_2, t_{n-1}}$. Under the same hypothesis as in Lemma \ref{X uniformly bounded}, the quantity $A_3(\rho)$ is uniformly bounded in $\rho$ and in $\omega \in \Omega$.
\end{lema}
\begin{proof}
    Applying the relation $\sigma_N U_n = I_n + \rho R_n$ we readily see by applying norms that
    \[
    \sigma^{k_4}_N||U_n||^{k_4}_{k_2,p_2, t_{n-1}} = ||I_n + \rho R_n||^{k_4}_{k_2,p_2, t_{n-1}}.
    \]
   In virtue of Minkowski's inequality, the fact that $(|a|+|b|)^{k_4} \leq C(|a|^{k_4} + |b|^{k_4})$ and $\rho \leq 1$, we have
   \begin{equation} \label{bound Un}
  ||U_n||^{k_4}_{k_2,p_2, t_{n-1}} \leq C \sigma_N^{-k_4} \left( ||I_n||^{k_4}_{k_2,p_2, t_{n-1}} + ||R_n||^{k_4}_{k_2,p_2, t_{n-1}} \right),
   \end{equation}
   so it is enough to bound the quantities  $||I_n||^{k_4}_{k_2,p_2, t_{n-1}}$ and $||R_n||^{k_4}_{k_2,p_2, t_{n-1}}$ separately. Lemma \ref{bound for A_1} gives us the bound
    \begin{equation} \label{estimate 1}
    ||R_n||^{k_4}_{k_2, p_2, t_{n-1}} \leq C N^{-k_4/2H}.
    \end{equation}
    In particular, this bound implies that
    \begin{equation} \label{Bound Rn}
    \sigma^{-k_4}_N||R_n||^{k_4}_{k_2,p_2,t_{n-1}} \leq C N^{\frac{-k_4}{2}\left( \frac{1}{H}-1\right)}.
    \end{equation}
    For the $I_n$ term, we know that $D_sI_n = \sigma(X_{s-r})$ and therefore, the $j$-th order derivatives of $I_n$ in directions belonging to $[t_n, t_{n-1}]^j$ vanish if $j \geq 2$ due to the fact that the partition is chosen so that $s-r < t_{n-1}$ for all $s \in [t_{n-1}, t_n]$. This implies that
    \[
    ||I_n||^{k_4}_{k_2, p_2, t_{n-1}} = \left(E_{t_{n-1}}(|I_n|^{p_2}) + E_{t_{n-1}}(||DI_n||_{\mathcal{H}[t_{n-1},t_n]}^{p_2})  \right)^{k_4/p_2}.
    \]
    On the one hand, since $I_n = \int_{t_{n-1}}^{t_n} \sigma(X_{s-r})dB^H_s$ and $s-r < t_{n-1}$, the law of $I_n$ conditioned to $\mathcal{F}_{t_{n-1}}$ is conditionally Gaussian with zero mean and variance $||\sigma(X_{\cdot -r})||_{\mathcal{H}[t_{n-1},t_n]}^2$. Hence, it is then well-known that by the properties of the moments of Gaussian random variables we have
    \[
    E_{t_{n-1}}(|I_n|^{p_2}) = C||\sigma(X_{\cdot -r})||^{p_2}_{\mathcal{H}[t_{n-1},t_n]}.
    \]
    Finally, using that $\sigma(X_{s-r}) \leq \Lambda$ it is easy to find that
    \begin{equation} \label{bound variance}
    ||\sigma(X_{\cdot-r})||^{p_2}_{\mathcal{H}[t_{n-1},t_n]} \leq \Lambda^{p_2} \sigma_N^{p_2},
    \end{equation}
    from which we conclude that
    \[
    E_{t_{n-1}}(|I_n|^{p_2}) \leq C\sigma^{p_2}_N
    \]
    for a certain constant $C > 0$. Now, since $||DI_n||^{p_2}_{\mathcal{H}[t_{n-1},t_n]} = ||\sigma(X_{\cdot}-r)||^{p_2}_{\mathcal{H}[t_{n-1},t_n]}$, we resort to inequality \eqref{bound variance} to conclude that there exists $C > 0$ such that
    \[
    E_{t_{n-1}}(||DI_n||^{p_2}_{\mathcal{H}[t_{n-1},t_n]}) \leq C\sigma^{p_2}_N.
    \]
    All in all, the conclusion of the previous estimates is that
    \begin{equation} 
        ||I_n||^{k_4}_{k_2, p_2, t_{n-1}} \leq C \sigma_N^{k_4},
    \end{equation}
    and in particular,
    \begin{equation}\label{Bound In}
        \sigma_N^{-k_4}||I_n||^{k_4}_{k_2, p_2, t_{n-1}} \leq C.
    \end{equation}
   
    To end the proof of this lemma, we plug in estimates \eqref{Bound Rn} and \eqref{Bound In} in \eqref{bound Un} and we get
    \[
    ||U_n||^{k_4}_{k_2, p_2, t_{n-1}} \leq C
    \]
    for a constant $C > 0$ depending on $p_2, \Lambda, k_4$ and $H$.
\end{proof}
 Let us now illustrate how this 3 lemmas prove Proposition \ref{bound J_2n a priori}. Recall that, using the integration by parts formula, we derived the estimate
    \[
    |J_{2,n}| \leq \sigma^{-2}_N A_1 \int_0^1 A_2(\rho) A_3(\rho) d \rho.
    \]
First, the conclusion of Lemma \ref{bound for A_1} is that
\[
\sigma_N^{-2}A_1 \leq \frac{C N^{-\gamma}}{\sigma_N}
\]
for real constants $C, \gamma > 0$. Moreover, the conclusion of Lemmas \ref{bound for A2} and \ref{bound for A3} respectively is that there exists a constant $C > 0$ such that $A_2(\rho) \leq C$ and $A_3(\rho) \leq C$ uniformly in $\rho \in [0,1]$ and in $\omega \in \Omega$. This information lead us to the following estimate:
\begin{equation} \label{Final bound J2n}
|J_{2,n}| \leq \sigma^{-2}_N A_1 \int_0^1 A_2(\rho) A_3(\rho) d \rho \leq \frac{C N^{-\gamma}}{\sigma_N}
\end{equation}
which proves Proposition \ref{bound J_2n a priori}. With these bounds on $J_{1,n}$ and $J_{2,n}$, we have all the tools needed in order to prove the lower bound for $t > r$.

\subsection{Proof of the lower bound} \label{Final proof}
In the section \ref{First case} we have already proved the result when $t \in (0,r]$, and in subsections \ref{bound J1n} and \ref{bound J2n} we have proved all the auxiliary results that allow us to proof the lower bound. This section is devoted to put all the information together and conclude that the density of the solution $X_t$ of \eqref{e: main equation} is strictly positive in its support.

\textit{Proof of Theorem \ref{t: main result}}
Recall that, thanks to Proposition \ref{fita inf J_1} and the bound we have concluded in equation \eqref{Final bound J2n}, we can bound $E_{t_{n-1}}[\delta_x(F_n)]$ as
\begin{equation} \label{lower bound conditioned}
E_{t_{n-1}}[\delta_x(F_n)] \geq \frac{1}{\sqrt{2\pi \Lambda^2 \sigma^2_N}} \exp \left( - \frac{(x-F_{n-1})^2}{2 \lambda^2 \sigma^2_N} \right) - \frac{C N^{-\gamma}}{\sigma_N}.
\end{equation}
We define the intervals $I_n = I(y_n, c_1 \sigma_N) := \{z \in \R; |z-y_n|< c_1\sigma_N\}$ for some constant $c_1 > 0$ to be determined and $y_n = \eta_0 + \frac{n}{N}(x-\eta_0)$ where $y_0 = \eta_0 = F_0$ is the initial condition of the SDDE. We also define $\{x_n; n = 0, \dots, N\}$ where $x_0  = \eta_0 = F_0$, $x_j \in I_j$ for $j = 1, \dots, N-1$ and $x_N = x$. Using the properties of the conditional expectation and the estimate \eqref{lower bound conditioned} we find that
\[
\E[\delta_x(F_N)] =  \E[E_{t_{N-1}}[\delta_x(F_N)]] \geq \frac{c}{\sigma_N}\E\left[\exp \left( - \frac{(x-F_{N-1})^2}{2 \lambda^2 \sigma^2_N} \right) -  C N^{-\gamma}\right]
\]
for a certain $\gamma > 0$. Hence,
\[
\begin{split}
    &\E[\delta_x(F_N)]  \\
    \geq & \frac{c}{\sigma_N} \int_{\R}\E\left[\left(\exp \left( - \frac{(x-F_{N-1})^2}{2 \lambda^2 \sigma^2_N} \right) -  C N^{-\gamma}\right)\delta_{x_{N-1}}(F_{N-1})\right]dx_{N-1} \\
    \geq & \frac{c}{\sigma_N} \int_{I_{N-1}}\E\left[\left(\exp \left( - \frac{(x-F_{N-1})^2}{2 \lambda^2 \sigma^2_N} \right) - C N^{-\gamma}\right)\delta_{x_{N-1}}(F_{N-1})\right]dx_{N-1}.
\end{split}
\]
Since we want this integral to be non-zero (otherwise, the resulting bound is $p_t(x) \geq 0$ which gives us no information) we need the following compatibility conditions to be satisfied:
\begin{enumerate}
     \item $x_{N-1} \in I_{N-1}$,
     \item $|x_{N-1} - F_{N-1}| \leq c_1\sigma_N$.
\end{enumerate}
Notice that if we choose $N$ of the form $N \geq \frac{|x-\eta_0|^2}{4c_1^2 t^{2H}}$, then we have 
\begin{equation} \label{x-F(N-1)}
    |x-F_{N-1}|\leq 4c_1 \sigma_N.
\end{equation}
We define a constant $c_2 \geq \frac{1}{4c_1^2}$ such that
\[
N = \frac{c_2|x-\eta_0|^2}{t^{2H}}.
\]
(Notice that for all $c_2 \geq \frac{1}{4c_1^2}$, the condition \eqref{x-F(N-1)} is satisfied). The idea is that $c_2$ is a constant that controls how big we want $N$ to be, so the larger $c_2$ is, the larger $N$ is. We choose $c_1$ small enough such that $\exp(-\frac{8c_1^2}{\lambda^2}) \geq 1/2$ and, for the moment, $c_2$ is chosen so that $N = \frac{c_2 |x-\eta_0|^2}{t^{2H}}$ satisfies $C N^{-\gamma} \leq 1/4$. Under this first restrictions, we get on the one hand
\begin{equation} \label{lower bound exponential}
\exp \left(- \frac{(x-F_{N-1})^2}{2\lambda^2 \sigma^2_N} \right) \geq \exp \left(- \frac{16c_1^2 \sigma_N^2}{2\lambda^2 \sigma^2_N} \right)  =\exp \left(- \frac{8c_1^2}{\lambda^2} \right) \geq \frac{1}{2}.
\end{equation}
On the other hand, using that $CN^{-\gamma} \leq 1/4$, we have

\begin{align*}
     &\frac{c}{\sigma_N} \int_{I_{N-1}}\E\left[\left(\exp \left( - \frac{(x-F_{N-1})^2}{2 \lambda^2 \sigma^2_N} \right) - C N^{-\gamma}\right)\delta_{x_{N-1}}(F_{N-1})\right]dx_{N-1}\\
     \geq &\frac{c}{\sigma_N} \int_{I_{N-1}} \E \left[ \left( \frac{1}{2} - \frac{1}{4}\right)\delta_{x_{N-1}}(F_{N-1}) \right] dx_{N-1},
\end{align*}
so we are lead to
\[
\E[\delta_x(F_N)] \geq \frac{c}{4\sigma_N} \int_{I_{N-1}} \E[\delta_{x_{N-1}}(F_{N-1})] dx_{N-1}.
\]
In particular, this constraints imply that $E_{t_{n-1}}[\delta_{x_{n}}(F_n)] \geq \frac{c}{4\sigma_N}$ for all $n = 1, \dots, N$. Indeed, by equation \eqref{lower bound conditioned} we have
\[
\begin{split}
E_{t_{n-1}}[\delta_{x_{n}}(F_n)] \geq &\frac{c}{\sigma_N} \left[\exp \left( - \frac{(x_{n}-F_{n-1})^2}{2 \lambda^2 \sigma^2_N} \right) - C N^{-\gamma}\right]\\
\geq &\frac{c}{\sigma_N} \left[ \exp \left( - \frac{(x_{n}-F_{n-1})^2}{2 \lambda^2 \sigma^2_N} \right) - \frac{1}{4} \right]
\end{split}
\]
Now, using the fact that $$|x_n - F_{n-1}| \leq |x_n - y_n| + |y_n - y_{n-1}| + |y_{n-1} - x_{n-1}| + |x_{n-1} - F_{n-1}| \leq 4c_1 \sigma_N$$ if $|y_n - y_{n-1}| \leq c_1 \sigma_N$, we have that
\[
\exp \left( - \frac{(x_{n}-F_{n-1})^2}{2 \lambda^2 \sigma^2_N} \right) \geq \exp \left( - \frac{16c_1^2 \sigma_N^2}{2 \lambda^2 \sigma^2_N} \right) \geq \frac{1}{2}.
\]
It is necessary to remark that the new constraint $|y_n - y_{n-1}| \leq c_1 \sigma_N$ is not restrictive. Indeed, in order to get a non-trivial lower bound, the condition $|y_n - y_{n-1}| \leq c_1 \sigma_N$ will appear naturally. Notice that, iterating one step from $N-1$ to $N-2$ and using that $\delta_{x_{N-2}}(F_{N-2})$ is $\mathcal{F}_{t_{N-2}}$-measurable we get
\[
\begin{split}
    \E[\delta_x(F_N)] \geq &\frac{c}{4\sigma_N} \int_{I_{N-1}}\E[\delta_{x_{N-1}}(F_{N-1})] dx_{N-1}\\
    \geq &\frac{c}{4\sigma_N}\int_{I_{N-1}} \int_{I_{N-2}} \E[\delta_{x_{N-1}}(F_{N-1}) \delta_{x_{N-2}}(F_{N-2})] dx_{N-2} dx_{N-1} \\
    = & \frac{c}{4\sigma_N}\int_{I_{N-1}} \int_{I_{N-2}} \E\left[ E_{t_{N-2}}[\delta_{x_{N-1}}(F_{N-1})] \delta_{x_{N-2}}(F_{N-2})\right] dx_{N-2} dx_{N-1} \\
    \geq & \left(\frac{c}{4\sigma_N}\right)^2 |I_{N-1}|\int_{N-2} \E[\delta_{x_{N-2}}(F_{N-2})]dx_{N-2}
\end{split}
\]
as long as $y_{N-1}, y_{N-2} \in I_{N-1}\cap I_{N-2} \neq \emptyset$, which happens if $|y_{N-1} - y_{N-2}| \leq c_1 \sigma_N$ (as mentioned in the previous paragraph, the condition $|y_n - y_{n-1}| \leq c_1 \sigma_N$ would arise naturally). From now on, we will assume that $|y_{n} - y_{n-1}|\leq c_1 \sigma_N$ for all $n = 1, \dots N$ and then we will check that with our choice of $N$, $c_1$ and $c_2$ this condition is satisfied. One can check by iterating this formula from $N-1$ to $1$ and using that all intervals $I_j$ have te same length as the interval centered at the origin with radius $c_1 \sigma_N$ that
\[
\begin{split}
\E[\delta_x(F_N)] \geq &\left(\frac{c}{4\sigma_N}\right)^N |I(0,c_1\sigma_N)|^{N-1} \\
= &\left( \frac{c}{4}\right)^N \left(\frac{N^{1/2}}{t^{H}} \right)^N \left(\frac{t^{H}} {N^{1/2}} \right)^{N-1} c_1^{N-1} \\
= & \frac{N^{1/2}}{t^{H}} \left( \frac{c}{4}\right)^N c_1^{N-1} \\
 = &\frac{1}{c_1 t^H}\exp\left( N\log\left( \frac{c \cdot c_1}{4}\right)+\frac{1}{2}\log(N)\right).
\end{split}
\]
To the previous restriction on $c_1$, that was $\exp\left(\frac{-8c_1^2}{\lambda^2} \right) \geq \frac{1}{2}$, we also impose that
\[
\rho = -\log\left( \frac{c \cdot c_1}{4}\right) > 0,
\]
and
\[
N \rho^N \geq 1.
\]
Now, taking into account that $N = \frac{c_2 |x-\eta_0|^2}{t^{2H}}$, and using that $N \geq 1$ (because we want $N$ to be a natural number) we get that
\[
\E\left[ \delta_x(F_N)\right] \geq \frac{1}{c_1 t^H}\exp \left( -\frac{\rho c_2 |x-\eta_0|^2}{t^{2H}}\right)
\]
which is a bound consistent with the one in \ref{e: bound main result rT}. The only work left to do is adjusting the constant $c_2$ so that it is compatible with all of the previous constraints. On the one hand, we have to take $c_2$ so that
\[
|y_{n+1}-y_n| \leq c_1 \sigma_N.
\]
Now, from the definition of the $y_n$'s, we have
\[
|y_{n+1}-y_n| = \frac{|x-\eta_0|}{N} = \frac{|x-\eta_0|}{\sqrt{N}}\frac{1}{\sqrt{N}}
\]
and, using that $N = \frac{c_2|x-\eta_0|^2}{t^{2H}}$ we get
\[
 |y_{n+1}-y_n| = \frac{|x-\eta_0|}{\sqrt{N}}\frac{1}{\sqrt{c_2}} \frac{t^H}{|x-\eta_0|} = \frac{\sigma_N}{\sqrt{c_2}}.
\]
Hence, it is enough to consider $c_2$ in such a way that
\[
\frac{1}{\sqrt{c_2}} \leq c_1, \quad c_2 \geq \frac{1}{4c_1^2}
\]
where the first constraint follows from the fact that we need $|y_{n}-y_{n-1}| \leq c_1 \sigma_N$ and the second constraint follows from relation \eqref{x-F(N-1)}. Choosing therefore an arbitrarily large $c_2$ satisfying
\[
\frac{1}{\sqrt{c_2}}\leq c_1
\]
is enough to conclude the bound. Finally, renaming then the constants, we obtain the desired bound \ref{e: bound main result rT}.

As a corollary of this result, we can state the conclusion of this work.
\begin{corol}
    Let $X$ be the solution to \eqref{e: main equation}. Under the same hypothesis as in Theorem \ref{t: main result}, the density function $p_t(x)$ of $X_t$ is strictly positive in its support.
\end{corol}
\begin{proof}
    Is a direct consequence of Theorem \ref{t: main result}.
\end{proof}
\appendix
\section{Real Analysis tools}
\begin{lema} \label{Ap 1}
    Let $\mathcal{H}$ be a real Hilbert space. Denote by $\langle \cdot, \cdot \rangle$ the scalar product in $\mathcal{H}$ and $||\cdot||$ its norm. Then, for any $f,g \in \mathcal{H}$, we have
    \[
    ||f+g||^2 \geq \frac{||f||^2}{2}-||g||^2
    \]
\end{lema}
\begin{proof}
    Notice that
    \begin{equation} \label{a: 1}
    \langle \frac{1}{\sqrt{2}}f + \sqrt{2}g, \frac{1}{\sqrt{2}}f + \sqrt{2}g \rangle = \frac{1}{2}||f||^2 + 2||g||^2 + 2\langle f, g \rangle \geq 0.
    \end{equation}
    Then, using that
    \[
    ||f + g||^2  = ||f||^2 + ||g||^2 + 2\langle f, g\rangle,
    \]
    equation \eqref{a: 1} turns into
    \[
    ||f+g||^2 - \frac{||f||^2}{2} + ||g||^2 \geq 0,
    \]
    which proves the inequality.
\end{proof}

\begin{lema} \label{Ap 2}
    Let $0<a<b$ be two positive real constants. Then,
    \begin{equation} \label{a: 2}
    \int_{a}^b \int_a^b |u-v|^{2H-2} du dv = c_H |b-a|^{2H}
    \end{equation}
\end{lema}
\begin{proof}
   Notice that the left hand side of \eqref{a: 2} can be expressed as
    \[
    \int_{a}^b \int_a^b |u-v|^{2H-2} du dv =\int_{0}^T \int_0^T \1_{[a,b]}(u) \1_{[a,b]}(v)|u-v|^{2H-2} du dv =\frac{1}{\alpha_H} \langle \1_{[a,b]}, \1_{[a,b]} \rangle_{\mathcal{H}}.
    \]
    Recall that
    \[
    \langle \1_{[0,a]},\1_{[0,b]} \rangle_{\mathcal{H}} = R_H(a,b)
    \]
    and $R_H(a,a) = a^{2H}$. Now, using that $\1_{[a,b]}(u) = \1_{[0,b]}(u) - \1_{[0,a]}(u)$ we have that
    \[
    \langle \1_{[a,b]}, \1_{[a,b]} \rangle_{\mathcal{H}} = \langle \1_{[0,b]}-\1_{[0,a]}, \1_{[0,b]} - \1_{[0,a]} \rangle_{\mathcal{H}} = R_H(b,b) + R_H(a,a) - 2R_H(a,b).
    \]
    Using now the definition of $R_H$, we find that
    \[
    R_H(b,b) + R_H(a,a) - 2R_H(a,b) = b^{2H} + a^{2H} - b^{2H} - a^{2H} + |b-a|^{2H} = |b-a|^{2H},
    \]
    so equality \eqref{a: 2} holds.
\end{proof}
\bibliographystyle{alpha}
\bibliography{references.bib}
\end{document}